\documentclass[11pt]{amsart}
\usepackage{geometry}
\usepackage{amssymb}

\usepackage{enumerate, hyperref,url,amssymb,amsmath,amsthm,amsxtra,mathtools,mathrsfs,calc,nccmath,color}
\usepackage{algorithm}

\usepackage{algpseudocode}
\usepackage{calc}
\usepackage{graphicx}
\usepackage{caption}
\usepackage{subcaption}
\usepackage{color}
\usepackage{needspace}

\newcommand{\lp}{\left(}
\newcommand{\rp}{\right)}
\newcommand{\Z}{\mathbb{Z}}
\newcommand{\Q}{\mathbb{Q}}

\newcommand{\ord}{\operatorname{ord}}
\newcommand{\SL}{\operatorname{SL}}
\newcommand{\GL}{\operatorname{GL}}

\newtheorem{thm}{Theorem}[section]

\newtheorem{lem}[thm]{Lemma}
\newtheorem{prop}[thm]{Proposition}
\newtheorem{example}[thm]{Example}
\newtheorem{conjecture}[thm]{Conjecture}
\newtheorem*{remark}{Remark}

\newcommand{\sk}{\big|_k }

\newcommand{\pMatrix}[4]{\left(\begin{matrix}#1 & #2 \\ #3 & #4\end{matrix}\right)}
\renewcommand{\pmatrix}[4]{\left(\begin{smallmatrix}#1 & #2 \\ #3 & #4\end{smallmatrix}\right)}

\newcommand{\tx}{\text}

\renewcommand{\(}{\left(}
\renewcommand{\)}{\right)}
\newcommand\be{\begin{equation}}
\newcommand\ee{\end{equation}}
\newcommand\bee{\begin{equation*}}
\newcommand\eee{\end{equation*}}

\theoremstyle{remark}

\theoremstyle{definition}

\begin{document}
\title{Theta-type congruences for colored partitions}


\author{O. Beckwith, A. Caione, J. Chen, M. Diluia, O. Gonzalez, J. Su }
\address{}
\curraddr{}
\email{}
\thanks{}

\date{}

\maketitle
\begin{abstract}
We investigate congruence relations of the form $p_r(\ell m n + t) \equiv 0 \pmod{\ell}$ for all $n$, where $p_r(n)$ is the number of $r$-colored partitions of $n$ and $m,\ell$ are distinct primes. 
\end{abstract}

\section{Introduction}
A partition of an integer $n$ is a non-increasing sequence of positive integers (called the ``parts" of the partition) whose sum is $n$. The partition function $p(n)$ is defined as the number of partitions of $n$.  In 1921, Ramanujan \cite{ram_cong_mathz} showed that for all $n \in \mathbb{Z}$,

\begin{gather}\label{eq:RamanujanCongruences}
p(5n+4) \equiv 0 \pmod{5} \\
p(7n+5) \equiv 0 \pmod{7} \\
p(11n+6) \equiv 0 \pmod{11}.
\end{gather}
Ramanujan's discoveries sparked interest among number theorists in the arithmetic properties of $p(n)$, as well as similar partition functions. An $r$-colored partition of $n$ is a partition of $n$ for which each part is assigned one of a fixed set of $r$ colors.  We let $p_r(n)$ denote the number of $r$-colored partitions of $n$ for positive integers $n$.  Congruences for $p_r(n)$ modulo powers of 11 similar to Ramanujan's congruences were found by Gordon \cite{gordon}. Kiming and Olsson \cite{KimingOlsson}, Boylan \cite{boylan-exceptional}, and Dawsey and Wagner \cite{DawseyWagner} discovered families of congruences of the form \begin{equation}\label{eq:RamType}
p_r(\ell n +t )\equiv 0 \pmod{\ell}
\end{equation}
for primes $\ell$. We refer to congruences of the form (\ref{eq:RamType}) where $\ell$ is prime as \emph{Ramanujan-type congruences}.  



Ahlgren and Boylan \cite{Ahlgren-Boylan} showed that $p(n)$ has no Ramanujan-type congruences other than (\ref{eq:RamanujanCongruences})-(3). On the other hand, the work of Ono and Ahlgren (\cite{Ono_annals}, \cite{Ahlgren-Ono}) shows that $p(n)$ has many congruences of the form $p(An+t) \equiv 0 \pmod{\ell}$. Treneer \cite{Treneer_1} extended this to all integer sequences which arise as Fourier coefficients of a weakly holomorphic modular form, including the functions $p_r(n)$. These results show that for any $r \in \mathbb{Z}$ and prime $\ell \ge 5$, there are infinitely many primes $Q$ such that for some $t \in \mathbb{Z}$, one has
\begin{equation*}
p_r(\ell Q^4n+t) \equiv 0 \pmod{\ell}
\end{equation*}
for all $n\in \mathbb{Z}$. Despite all this progress, there is no known explicit classification of all the congruences of $p_r(n)$. Finding new types of congruences, and showing various types of congruences do not exist, remains an active area of research.  

Ahlgren, Raum, and the first author \cite{ABR} investigated congruences of the form
\begin{equation}\label{eq:squarefree}p_r(\ell mn+t) \equiv 0 \pmod{\ell} \end{equation} 
with $r=1$ and $m$ cube-free and co-prime to $\ell$. Of course, if $\ell \in \{5,7,11 \}$ and $r=1$, then one has (\ref{eq:squarefree}) for any $m$ by the Ramanujan Congruences (\ref{eq:RamanujanCongruences})-(3).  When $r=1$, \cite{ABR} shows that there are no non-trivial instances of \ref{eq:squarefree} where $\ell ,m < 10,000$ and $m$ is prime, and they proved that for any $\ell \ge 5$, congruences of the type (\ref{eq:squarefree}) are rare in a precise sense unless a certain modular form has many coefficients divisible by $\ell$.
 
 We investigate congruences of the form (\ref{eq:squarefree}) for $r>1$ and $m$ prime, and we call such congruences \emph{theta-type congruences}. We say that a theta-type congruence is nontrivial if it does not follow from a Ramanujan-type congruence. Keeping with the method of \cite{ABR} and other papers on partition congruences, we use the following generating functions:
 $$
f_{r,\ell, 0} (z) := \sum_N p_r \lp \frac{\ell N +r}{24} \rp q^{\frac{N}{24}} =  q^{ - \frac{r \ell}{24}} \sum_{n \in \Z} p_r \left(\ell n -  \frac{r(\ell^2 -1)}{24} \right) q^{n}
$$
and
$$
f_{r,\ell, -1} (z) :=\sum_{\lp \frac{-rN}{\ell} \rp = -1} p_r \lp \frac{ N +r}{24} \rp q^{\frac{N}{24}} = \sum_{\substack{0 \le t \le \ell -1 \\ \lp \frac{r-24t}{\ell} \rp \lp \frac{r}{\ell} \rp = -1 }} q^{t - \frac{r}{24}} \sum_{n \in \Z} p_r(\ell n +t) q^{\ell n},
$$
where throughout, we set $p_r(x) := 0$ if $x$ is not a non-negative integer and $p_r(0):=1$. 

We find that, in contrast to the $r=1$ situation, there many nontrivial theta-type congruences for several values of $r>1$.
\begin{thm}\label{thm:main}
For $r_1, \ell$ as in Table \ref{Tab:etacongs}, we have either the Ramanujan-type congruence \begin{equation}\label{eq:Ram1}p_{r_1} \left(\ell n -  \frac{r_1(\ell^2 -1)}{24} \right) \equiv 0 \pmod{\ell}, \end{equation}  or we have a nontrivial theta-type congruence for every prime $m \neq \ell $ and integer $t$ satisfying $t \equiv r_1/24 \pmod{\ell}$ and 
$$
\left( \frac{24t-r_1}{m} \right) = - \left( \frac{\ell}{m} \right).
$$
Similarly, for $r_2, \ell$ as in Table \ref{Tab:etacongs}, we have either (\ref{eq:Ram1}), or we have a nontrivial theta-type congruence for every prime $m \nmid 3\ell $ and integer $t$ satisfying $t \equiv r_2/24 \pmod{\ell}$ and 
$$
\left( \frac{24t-r_2}{m} \right) = - \left( \frac{3\ell}{m} \right).
$$

\end{thm}
\begin{remark}
Out of the 775 pairs $(r_1,\ell)$ with $5 \le \ell \le 349$ in Table \ref{Tab:etacongs}, we have a Ramanujan-type congruence for 3. The other 772 pairs produce infinitely many nontrivial theta-type congruences. See the Mathematica \cite{Mathematica} code on the first author's webpage for details. \end{remark}
\begin{example}
For $\ell =7$, $r=17$, we do not have (\ref{eq:Ram1}) because $p_{17} ( 1) = 17 \not\equiv 0 \pmod{7}$. Theorem \ref{thm:main} implies that we have (\ref{eq:squarefree}) for every prime $m \neq 7$ and $t \equiv 1 \pmod{7}$ such that 
$$
\left( \frac{24t-17}{m} \right) = - \left( \frac{7}{m} \right).
$$
\end{example}

\begin{center}
\captionof{table}{\label{Tab:etacongs}}
 \begin{tabular}{||c | c| c||} 
 \hline
$\ell$ & $r_1$ & $r_2$  \\
 \hline
$\ell \equiv 1 \pmod{6}$,   $\ell \ge 5$ & $r_1 = 3 \ell - 4$    & $r_2 =  \ell - 4$   \\
\hline
 $\ell \equiv 1 \pmod{4}$, $\ell \ge 5$  & $r_1 = 5 \ell -6$ & $r_2 = 3 \ell -6$ \\
\hline
$\ell \equiv 1 \pmod{3}$, $\ell \ge 5$ & $r_1 = 7 \ell -8$ & $r_2 = 5 \ell -8$  \\
\hline
$\ell \equiv 1 \pmod{12}$, $\ell \ge 5$ & $r_1 = 9 \ell -10$  & $r_2 = 7 \ell -10$  \\
\hline
 $\ell \ge 5$ & $r_1 = 11 \ell -12$  & $r_2 = 9 \ell -12$  \\
\hline
 $\ell \equiv 1 \pmod{12}$, $\ell \ge 7$ & $r_1= 13 \ell -14$   & $r_2 = 11 \ell -14$ \\
\hline
$\ell \equiv 1 \pmod{3}$, $\ell \ge 7$ & $r_1  = 15 \ell -16$   & $r_2 = 13 \ell -16$ \\
\hline
$\ell \equiv 1 \pmod{4}$, $\ell \ge 11$ & $r_1 = 17 \ell -18$ & $r_2 = 15 \ell -18$  \\
\hline
$\ell \equiv 1 \pmod{6}$, $\ell \ge 11$ & $r_1 = 19 \ell -20$  & $r_2 = 17 \ell -20$\\
\hline
$\ell \equiv 1 \pmod{12}$, $\ell \ge 11$ & $r_1 = 21 \ell -22$    & $r_2 = 19 \ell -22$  \\
\hline
$\ell \ge 11$  & $r_1 = 23 \ell -24$  & $r_2 = 23 \ell -24$. \\
\hline
\end{tabular}
\end{center}

\begin{remark}
For $r=\ell-4$, Theorem 2.1 (2) \cite{boylan-exceptional} shows $f_{r,\ell,0} \equiv 0 \pmod{\ell}$ if $\ell \not\equiv 1 \pmod{6}$, and a succinct proof of this fact is also given in \cite{Dicks} (see Lemma 7.1). The method of \cite{Dicks} also gives a short proof of the $r=\ell -4$ case in Theorem \ref{thm:main}.
\end{remark}

Theorem \ref{thm:main} is a consequence of the fact that for $r,\ell$ in Table \ref{Tab:etacongs}, $f_{r,\ell,0}$ is congruent modulo $\ell$ to a theta function. More precisely, we let $\eta$ denote the Dedekind $\eta$ function:
\begin{gather*}
  \eta(z):=q^\frac1{24}\prod_{n=1}^\infty(1-q^n),
\end{gather*}
where $z$ lies in the upper half plane $\mathbb{H}$, and we let $q^n = e^{2 \pi i n z}$ for all $n \in \mathbb{Q}$.
Then we have the classical identities 
\begin{equation}\label{eq:eta3identity}
\eta^3 (z) = \sum_{n =1}^{\infty} \left( \frac{-4}{n} \right)n q^{n^2/8}
\end{equation}
and
\begin{equation}\label{eq:etaidentity}
\eta(z) = \sum_{n=1}^{\infty} \left( \frac{12}{n} \right) q^{n^2/24}.
\end{equation}
Theorem \ref{thm:main} is an immediate consequence of the following theorem:

\begin{thm}\label{thm:etafamily}
Let $\ell \ge 5$ be prime, and let $r$ be a positive integer co-prime to $\ell$. 
\begin{enumerate}
\item If $3 \le a \le \operatorname{min}\{21,2\ell+1\}$ is odd and $r$ is given by $r = a(\ell-1)-1$, then if $\ell \equiv 1 \pmod{24/(24,a+1)}$, we have
 $$f_{r,\ell,0} \equiv p_r \left(\frac{\ell + r }{24} \right) \cdot \eta \pmod{\ell}.$$
 \item  If $1 \le a \le \operatorname{min}\{19,2\ell+1\}$ is odd and $r$ is given by $r = a(\ell-1)-1$, then if $\ell \equiv 1 \pmod{24/(24,a+3)}$, we have
 $$f_{r,\ell,0} \equiv p_r \left(\frac{3\ell + r }{24} \right) \cdot \eta^3 \pmod{\ell}.$$
 \end{enumerate}
\end{thm}

There are theta-type congruences that don't follow from Theorem \ref{thm:etafamily}. The next theorem describes a few of them. To state the conditions, we need a couple definitions. For integer sequences $a(n)$, we let 
$$
\ord_{\ell} \left( \sum_{n \in \mathbb{Z}} a(n) q^{n/24} \right) := \operatorname{inf} \{ n \in\mathbb{Z}: \ell \nmid a(n) \}.
$$

We let $b(r,\ell)$ be given by
$$
b(r,\ell) := (\ell - 1) \Big\lfloor \frac{1}{\ell -1} \left( \ell + \frac{r (\ell^2- 1 ) - 2}{2 \ell } \right) \Big\rfloor - \frac{r \ell }{2}. 
$$

\begin{thm}\label{thm:abnormal}
Let $\ell >3$ be prime, and let $r$ be a positive integer not divisible by $\ell$. 
\begin{enumerate}
\item If $b(r,\ell) = \ell/2$, $\ord_{\ell}(f_{r,\ell,0}) > \ell$, and $r  \equiv -1 \pmod{24}$ , then $f_{r,\ell,0} \equiv \alpha \eta^{\ell} \pmod{\ell}$ for some $\alpha \in \mathbb{Z}$. 
\item If  $\ell^2 = r+4$, $f_{r, \ell, 0} \equiv 0 \pmod{\ell}$,  $\ord_{\ell}(f_{r,\ell,-1}) > 0$, and $r \equiv -3 \pmod{24}$, then $f_{r,\ell,-1} \equiv \alpha \eta^{3} \pmod{\ell}$ for some $\alpha \in \mathbb{Z}$.
\item If $\ell^2 = r+2$, $f_{r,\ell,0} \equiv \beta \eta^{\ell} \pmod{\ell}$ for some integer $\beta$,  $\ord_{\ell}(f_{r,\ell,-1}) > 0$, and $r \equiv -1 \pmod{24}$, then $f_{r,\ell,-1} \equiv \alpha \left( \eta - \left( \frac{12}{\ell} \right) \eta^{\ell^2}  \right) \pmod{\ell}$ for some $\alpha \in \mathbb{Z}$.
\end{enumerate}
\end{thm}

%
%

\begin{example} 
Theorem \ref{thm:abnormal} implies that $f_{23,7,0} \equiv  \alpha \eta^7 \pmod{7}$ for some integer $\alpha$. Comparing the leading coefficients, we see that we can take $\alpha = p_{23}(3) \equiv 3 \pmod{7}$. 

As an example of the second case, we have $f_{21,5,-1} \equiv  \eta^3 \pmod{\ell}$.

The third case of Theorem \ref{thm:abnormal} implies that $f_{23,5,-1} \equiv  \eta^{5^2} + \eta \pmod{5}$. 

More examples are given in Table \ref{Tab:abnormalexamples}.
\end{example}

\begin{center}
\captionof{table}{Examples of each type in Theorem \ref{thm:abnormal}. \label{Tab:abnormalexamples}}
 \begin{tabular}{||c || c || c || } 
 \hline
Case &  Function & Nonzero Examples $(r,\ell)$ \\
 \hline
1   &   $\eta^{\ell}$  &   (23,5), (23,7), (47,7), (47,13), (71,13), \\
    &                        &(71,19), (95,13), (95,17), (119,11), (119,13)  \\
 \hline
2  &$\eta^3$ & (21,5), (45,7) \\
\hline

3  & $\left( \frac{12}{\ell} \right) \eta^{\ell^2} - \eta$ & (23,5).\\
  \hline
  \end{tabular}
  \end{center}

 Numerical evidence suggests that in fact (\ref{eq:squarefree}) occurs \emph{only} when $f_{r,\ell,\delta}$ is congruent modulo $\ell$ to a theta function. 
\begin{conjecture}\label{thm:numericalevidence}
Let $r>0$ be odd and $\ell \ge 5$ be a prime not dividing $r$.
We have a theta-type congruence for some integer $t$ with $\left( \frac{r (r-24t)}{\ell} \right) = \delta \in \{0,-1\}$ if and only if 
$$f_{r,\ell,\delta} \equiv \sum_{n = 0}^{\infty} a(n) q^{bn^2/24} \pmod{\ell} $$ 
for some integers $b$ and $a(n)$.
\end{conjecture}

Our conjecture is supported by the numerical evidence summarized in Theorem \ref{thm:scarcity}. 
Modifying an algorithm in \cite{ABR}, we show that there are no nontrivial theta-type congruences for thousands of triples $(r, \ell, m)$. 
\begin{thm}
\label{thm:scarcity}
For odd $r$ such that $1 \le r < 24$, there are no nontrivial theta-type congruences with $\ell$ and $m$ in the range $[5, 6133]$ such that $\ell \nmid r$, except for those in Table \ref{Tab:apparentcongs}. For the triples in Table \ref{Tab:apparentcongs}, the generating function $f_{r,\ell,\delta}$ is congruent to a nonzero multiple of the theta function given in the corresponding row.
\end{thm}

\begin{center}
\captionof{table}{The pairs $(\ell,\delta)$ for each odd $1 < r < 24$ for which nontrivial theta-type congruences exist. \label{Tab:apparentcongs}}
 \begin{tabular}{||c || c ||} 
 \hline
 Function & $(r,\ell,\delta)$  \\
 \hline
 \hline
 $\eta$ & $(17,7,0), (19,5,0) $ \\
 \hline
 $\eta^3$ & $(3,7,0), (9,13,0), (9,5,0), (15,19,0), (21,5,-1)$ \\
 \hline
 $\eta^{\ell}$ & $(23,5,0), (23,7,0)$ \\
 \hline
$\left( \frac{12}{\ell} \right) \eta^{\ell^2} - \eta$ & $(23,5,-1),(23,7,-1)$ \\
\hline
\end{tabular}
\end{center}


It would be interesting to know more about the arithmetic properties of triples $(r, \ell, m)$ appearing in theta-type congruences. Generalizing a result in \cite{ABR}, we show that the prime factors of squarefree $m$ appearing in nontrivial instances of (\ref{eq:squarefree}) are relatively prime to $r (r - 24 t)$:
\begin{thm}\label{thm:0class}
Let $\ell \ge 5$ be a prime, let $r$ be an odd integer relatively prime to $\ell$, and let $m = m' m''$ be a square-free integer relatively prime to $6\ell$, where $m'' | r (r -24t)$ and $(m', r (r -24t))= 1$. 
 
We have the congruence
\begin{equation}\label{eq:cong8}
p_r(\ell m  n +t) \equiv 0 \pmod{\ell},
\end{equation}
for all $n \in \mathbb{Z}$ if and only if we have the congruence
$$
p_r(  \ell m'n + t) \equiv 0 \pmod{\ell}
$$
for all $n \in \mathbb{Z}$.
\end{thm}

In Section 2, we present the background from the theory of modular forms, and prove a few lemmas expressing $f_{r,\ell,\delta}$ in terms of modular forms. Section 3 proves Theorem \ref{thm:0class}, Section 4 proves Theorem \ref{thm:scarcity},  and Section 5 proves Theorems \ref{thm:etafamily}, \ref{thm:main}, and \ref{thm:abnormal}. Section 6 provides some examples and questions for future consideration.

\section*{Acknowledgements}
The authors would like to thank Scott Ahlgren for some very helpful comments on an earlier version of this manuscript. This work was supported by the Illinois Geometry Lab. IGL research is supported by the Department of Mathematics at the University of Illinois at Urbana-Champaign. This material is based upon work supported by the National Science Foundation under Grant No. DMS-1449269. Any opinions, findings, and conclusions or recommendations expressed in this material are those of the authors and do not necessarily reflect the views of the National Science Foundation.

\section{Preliminaries}
In this section, we present the necessary background from the theory of modular forms.
\subsection{Modular forms}
Throughout, $\mathbb{H} := \{ z \in \mathbb{C}: \operatorname{Im}(z) > 0 \}$ denotes the upper half plane and $z$ represents an element of $\mathbb{H}$. We let $e(x) = e^{2 \pi ix}$ and $q^n:=e(n z)$ for $n \in \mathbb{Q}$. 

We will use the action of $ \GL_2^+(\Q)$ on $\mathbb{H}$ given by fractional linear transformation:
$$
\pMatrix abcd   \cdot z := \frac{az+b}{cz+d}.
$$

If $f : \mathbb{H} \to \mathbb{C}$ is a function on the upper half-plane,  $k\in \frac12 \Z$,  and $\gamma=\pmatrix abcd\in \GL_2^+(\Q)$, 
we define
\begin{gather*}
\(f\sk\gamma\)(\tau):=(\det\gamma)^\frac k2(c\tau+d)^{-k}f(\gamma \tau),
\end{gather*}
where we always choose the principal branch of the square root.

The Dedekind eta function defined in Section 1 satisfies the transformation law
$$
\eta |_{1/2} \gamma = \nu_{\eta} (\gamma) \eta
$$
for every $\gamma \in \SL_2(\mathbb{Z})$, with $\nu_{\eta}$ given explicitly in \cite[\S 4.1]{knopp}.

We will require the congruence subgroups
$$\Gamma_0(N) := \Big\{ \pMatrix{a}{b}{c}{d} \in \SL_2(\mathbb{Z}) : c \equiv 0 \pmod{N} \Big\}.$$

Let $\chi$ be a Dirichlet character modulo $N$. We say that a holomorphic function $f: \mathbb{H} \to \mathbb{C}$ is a \emph{modular form of weight $k \in \frac{1}{2} \mathbb{Z}$ and level $N$ with respect to the multiplier system $\chi \nu_{\eta}^r$} if 
\begin{equation}\label{eq:functionalequation}
f\sk \gamma=\chi(d) \nu_{\eta}^r(\gamma) f\ \ \ \text{for all}\ \  \  \gamma=\pmatrix abcd\in \Gamma_0(N)
\end{equation}
and if $f \sk \sigma$ has a Fourier expansion of the form 
$$
f \sk \sigma = \sum_{n =0}^{\infty} a_{\sigma} (n) q^{n/24} \ \ \ \text{for all}\ \  \  \sigma \in \SL_2(\mathbb{Z})
$$ 
Such a function is a \emph{cusp form} if $\lim_{y \to \infty} f( \sigma (x+iy)) = 0$ for all $\sigma \in \SL_2(\mathbb{Z})$. On the other hand, a  \emph{weakly holomorphic modular form of weight $k \in \frac{1}{2} \mathbb{Z}$ and level $N$ with respect to $\chi \nu_{\eta}^r$} is a holomorphic function $f : \mathbb{H} \to \mathbb{C}$ which satisfies (\ref{eq:functionalequation}) and the property that for each $\sigma \in \SL_2(\mathbb{Z})$, there is a polynomial $P_{\sigma}$ such that $\lim_{y \to \infty} f(\sigma z) - P_{\sigma} (q^{\frac{-1}{24N}})  = 0$.  We denote by $M_k\(N, \chi \nu_{\eta}^r \)$, $S_k\(N,  \chi \nu_{\eta}^r \)$, and $M^!_k\(N,  \chi \nu_{\eta}^r \)$ the complex vector spaces of modular forms, cusp forms, and weakly holomorphic modular forms of weight $k$ and level $N$ with respect to $\chi \nu_{\eta}^r$.

 It is easy to check that any $f\in M_k^! \lp N, \chi \nu_{\eta}^r \rp$ has a Fourier expansion of the form
\begin{equation}
\label{eq:fourier-expansion}
  f
=
  \sum_{n \equiv  r\pmod{24}}
  a(n)q^\frac n{24}.
\end{equation}



%
%

\subsection{Level 1 modular forms}

When $N=1$, we let $M_k  \( \chi \nu_{\eta}^r \) := M_k \( 1, \chi \nu_{\eta}^r \)$, $S_k \( \chi \nu_{\eta}^r \) := S_k(1, \chi \nu_{\eta}^r)$, and $M_k^!  \( \chi \nu_{\eta}^r \) := M_k^! \( 1, \chi \nu_{\eta}^r \)$. When $N=1$ and $24|r$, these are the usual spaces of integer weight modular forms of level one. In this case, we denote $M_k := M_k \( 1, \chi \nu_{\eta}^r \)$, $S_k := S_k(1, \chi \nu_{\eta}^r)$. We recall that for even $k \ge 0$, the dimension of $M_k$ is given by
\begin{equation}\label{eq:dimensionformula}
\operatorname{dim} (M_k) = \begin{cases} 
\lfloor \frac{k}{12}+1 \rfloor & k \not\equiv 2 \pmod{12}\\
\lfloor \frac{k}{12} \rfloor & k \equiv 2 \pmod{12}. \\
\end{cases}
\end{equation}
When $M_k$ is nontrivial, we have $\operatorname{dim}(S_k) = \operatorname{dim}(M_k) -1$.

We recall that the Eisenstein series $E_k \in M_k$ are given by
$$
E_k(\tau) = 1 + \frac{(2\pi i)^{k}}{ (k-1)! \zeta(k)} \sum_{n=1}^{\infty} \sigma_{k-1}(n) q^n. 
$$ 
For primes $\ell \ge 5$, we have
$$
E_{\ell -1} \equiv 1 \pmod{\ell}.
$$

We will also refer to the modular discriminant
$$
\Delta(\tau) = \eta^{24} \in S_{12}.
$$

Finally, we will need to use the \emph{Miller basis} in a few of our proofs: 
\begin{lem}[Lemma 2.20 \cite{Stein}]\label{thm:miller}
There is a unique basis $\{ f_1, f_2, \cdots, f_d \}$ of $S_k$ such that for each each $f_i =: \sum_{n=1}^{\infty} a_i(n) q^n$, we have $a_i(n) \in \mathbb{Z}$ and $a_i(j) = \delta_{i,j}$ for $1 \le i,j \le d$.
\end{lem}
The basis in Lemma \ref{thm:miller} is called the Miller basis of $S_k$.

\subsection{Linear operators}
We will make use of a few linear operators which act nicely on the spaces $M_k\left( N, \chi \nu_\eta^r \right)$ and $M_k^!\left( N, \chi \nu_\eta^r \right)$, specifically, the $U_m$ and $V_m$ operators, and twisting by a quadratic character. We define these operators here and state their properties, which can be found in Section 2 of \cite{ABR}. Some of these properties are stated in \cite{ABR} with the assumption $(r,24)=1$, but their proofs only depend on $r$ being odd. All can be verified by a straightforward calculation (see pages 128-133 of \cite{koblitz}). Throughout $r$ is a positive odd integer.

Suppose that $f \in M_k^! (N,\chi\nu_{\eta}^r )$ has Fourier expansion~\eqref{eq:fourier-expansion},  and that $m$ is a positive integer. 
Define 
\begin{equation}\label{eq:umdef}
  f\big|U_m:=\sum a(mn)q^\frac n{24}
\quad\text{and}\quad
  f\big|V_m:=\sum a(n)q^\frac {mn}{24}.
\end{equation}

  If~$M$ is an odd positive squarefree integer,  let $\chi_M = \big(\frac{\bullet}{M}\big)$ denote the quadratic character of modulus~$M$. Equations (2.8)-(2.9) of \cite{ABR} state that if  $Q \ge 5$ is prime, then
\begin{align}
\label{eq:uqaction}
  U_Q :& M_k^!\lp N, \chi \nu_\eta^r\rp \longrightarrow M_k^! \lp N \mfrac Q{(N, Q)},\, \chi\chi_Q^r \nu_\eta^{Qr}\rp,
\\
\label{eq:vqaction}
  V_Q :& M_k^!\lp N, \chi\nu_\eta^r\rp \longrightarrow M_k^! \lp NQ, \,\chi{\chi_Q}^r\nu_\eta^{Qr}\rp.
\end{align}
The $U_Q,V_Q$ operators act similarly on the spaces $M_k (N,\chi\nu_{\eta}^r )$ and $S_k (N,\chi\nu_{\eta}^r )$.

As in the equations in \cite{ABR} following (2.8)-(2.9), if $Q \ge 5$ is prime,  we define the twist of~$f \in M_k (N,\chi \nu_\eta^r)$ with Fourier expansion~\eqref{eq:fourier-expansion} by $\chi_Q$ as
\begin{equation}\label{eq:twistdef}
  f\otimes \chi_Q:=\sum \chi_Q (n) a(n)q^\frac n{24}
\tx{.}
\end{equation}
Then we have
\begin{equation}\label{eq:quadtwist}
  f\otimes\chi_Q \in M_k\left( NQ^2, \chi\nu_\eta^r\right).
\end{equation}

\subsection{Modular forms modulo $\ell$}
We rely on a few results from the theory of modular forms modulo primes. Throughout the subsection, $\ell \ge 5$ is prime and let $k \in \mathbb{Z}$.

We have the following consequence of a theorem of Swinnerton-Dyer (see \cite{SwinnertonDyer} or Proposition 2.43 of \cite{CBMS}):
\begin{lem}[\cite{SwinnertonDyer}]
If $f \in M_k$ and $g \in M_{k'}$ are such that $f \equiv g \pmod{\ell}$,  then $k \equiv k' \pmod{\ell -1}.$
\end{lem}

For $f \in M_k(\mathbb{F}_{\ell})$, Serre \cite{Serre} defined the \emph{filtration} $w(f)$ as  $$w(f) = \operatorname{inf} \{r \in \mathbb{Z}: f \in M_{r} ( \mathbb{F}_{\ell}) \}.$$

Let $\Theta := q \frac{d}{dq} $ be the Ramanujan Theta operator. We have the following results of Serre:

\begin{lem}[Serre \cite{Serre} Lemme 2]\label{thm:filtration}
For $\ell \ge 5$, we have 
$$w(f | U_{\ell}) \le \ell + \frac{w(f) -1 }{\ell},$$  
and
$$
w(f|\Theta) \le w(f) + \ell + 1.
$$
In the second inequality, there is equality if and only if $\ell \nmid w(f)$.
\end{lem}

\subsection{Previous results on congruences of colored partitions}
We requires a couple results from \cite{andersen}. The first says that congruences for $p_r(n)$ have a square class structure:
\begin{prop}[Lemma 5.3 of \cite{andersen}]\label{thm:squareclasses}
Suppose $p_r(mn+t) \equiv 0 \pmod{\ell}$ for all $n$, where $(m,24)=1$.  If $r- 24 t' \equiv (r -24 t) \cdot h^2 \pmod{m}$ and $(h,m)=1$, then $p_r(mn+t') \equiv 0 \pmod{\ell}$ for all $n$. 
\end{prop}

We also have the following restriction on congruences:
\begin{thm}[Theorem 1.4 of \cite{andersen}] 
Suppose $p_r(mn+t) \equiv 0 \pmod{\ell}$, and $\ell \nmid r$.  Then $\ell | m$ and
$$
 \lp \frac{r(r-24t)}{\ell} \rp \neq 1.  
$$
\end{thm}

\subsection{Modular properties of $f_{r,\ell,\delta}$}
The generating functions $f_{r,\ell,\delta}$ have been used in several papers on partition congruences going back to the work of Ono \cite{Ono_annals} and Ahlgren-Ono \cite{Ahlgren-Ono} for $r=1$,  more recently in \cite{ABR}, as well as \cite{Dicks} for $r > 1$. From these previous works, it is easy to see that $f_{r,\ell,\delta}$ is always a weakly holomorphic modular form, but we include a statement and proof of this fact here for completeness. 
\begin{lem}\label{thm:genfcn0}
We assume $\ell \ge 5$ is prime, and that $r$ is an odd integer co-prime to $\ell$. Then 
$$
f_{r,\ell,0} \equiv  \frac{\Delta^{\frac{r(\ell^2 - 1)}{24}} | U_{\ell}}{\eta^{r\ell}} \pmod{\ell}.
$$
It follows that $f_{r,\ell,0}$ is congruent modulo $\ell$ to a weakly holomorphic modular form in $M_{k_{r,\ell,0}}^! (\nu_{\eta}^{-r\ell})$, where
$$
k_{r,\ell,0} = \frac{r(\ell^2-\ell-1)}{2}.
$$
\end{lem}
\begin{proof}
Let $a(n)$ be the coefficients of $ \Delta^{\left( \frac{\ell^2 -1 }{24} \right) r}$. Then we have
\begin{align*}
\sum_n a(n) q^n &=: \Delta^{\left( \frac{\ell^2 -1 }{24} \right) r} \\
&= \eta^{r (\ell^2-1)} \\
&= \left( \prod_n (1- q^n)^{\ell^2 r} \right) \cdot q^{r\ell^2/24} \eta^{-r} \\
&= \left( \prod_n (1- q^n)^{\ell^2 r} \right) \cdot \sum_n p_r(n)q^{n + r \left( \frac{\ell^2 -1 }{24} \right)} \\
&= \left( \prod_n (1- q^n)^{\ell^2 r} \right) \sum_{n} p_r \left(n - r \left(\frac{\ell^2-1}{24}\right) \right) q^n.
\end{align*}
Applying the $U_{\ell}$ operator, we obtain
\begin{align*}
\sum_n a(\ell n) q^n &=: \Delta^{\left( \frac{\ell^2 -1 }{24} \right) r} | U_{\ell} \\
&\equiv (1 - q^n)^{\ell r} \sum_n p_r \left( \ell n - \left(\frac{\ell^2-1}{24} \right) r \right) \pmod{\ell}.
\end{align*}

Multiplying both sides of the last relation by $\eta^{-r\ell}$, we obtain
\begin{align*}
 \frac{ \Delta^{r (\ell^2 -1 )/24} | U_{\ell} }{\eta^{\ell r}} &\equiv \sum p_r (\ell n - r(\ell^2-1)/24 ) q^{n - (r\ell/24)}  \pmod{\ell} \\
&\equiv \sum p_r(\ell n - r(\ell^2-1)/24) q^{\frac{24 n - r\ell}{24}}  \pmod{\ell}  \\
&\equiv \sum_N p_r \left( \frac{\ell N + r}{24} \right) q^{N/24}  \pmod{\ell} \\
&\equiv f_{r,\ell,0} \pmod{\ell} .
\end{align*}

Finally, we observe that $\Delta^{\frac{r(\ell^2 - 1)}{24}} | U_{\ell} \equiv \Delta^{\frac{r(\ell^2 - 1)}{24}} | T_{\ell} \in M_{\frac{r(\ell^2 - 1)}{2}}$ and $\eta^{r\ell} \in S_{r\ell/2} (1, \nu_{\eta}^{r\ell})$, and it follows that their ratio lies in $M_{ \frac{r(\ell^2 - 1) - r\ell}{2}}^!(1, \nu^{-r\ell})$. 
\end{proof}

For $f_{r,\ell,-1}$, we have a similar result.

\begin{lem}\label{thm:genfcn1}
We assume $\ell \ge 5$ is prime, and that $r$ is an odd integer co-prime to $\ell$. Then 
$$
f_{r,\ell,-1} \in M_{k_{r,\ell,-1}}^! (1, \nu_{\eta}^{-r} ),
$$
where
$$
k_{r,\ell,-1}= \frac{r \ell (\ell^2- \ell - 1)}{2}.
$$

\end{lem}

\begin{proof}
Let $a(n)$ be as in the previous proof. Then we have
\begin{align*}
\left(\frac{-r}{\ell} \right)   \Theta^{\frac{\ell -1}{2}} \Delta^{r \left( \frac{ \ell^2 -1}{24} \right)} &\equiv  \left(\frac{-r}{\ell} \right) \Delta^{r \left( \frac{ \ell^2 -1}{24} \right)} \otimes \chi_{\ell} \pmod{\ell} \\
 &\equiv \left( \prod_n (1- q^n)^{\ell^2 r} \right) \sum_{n} \left( \frac{-rn}{\ell} \right) p_r \left(n - r \left(\frac{\ell^2-1}{24}\right) \right) q^n \pmod{\ell} .\\
 \end{align*}

Then it follows that
\begin{align*}
& \Delta^{\frac{r (\ell^2-1)}{24}} - \left(\frac{-r}{\ell} \right)  \Theta^{\frac{\ell -1}{2}} \Delta^{r \left( \frac{ \ell^2 -1}{24} \right)} \\
&\equiv \left( \prod_n (1- q^n)^{\ell^2 r} \right) \sum_{n} \left( 1 - \left( \frac{-rn}{\ell} \right)\right) p_r \left(n - r \left(\frac{\ell^2-1}{24}\right) \right) q^n \pmod{\ell} \\
 &\equiv \eta^{\ell^2 r} q^{-r\ell^2/24} \left( \sum_{\ell | n} p_r \left(n - r \left(\frac{\ell^2-1}{24}\right) \right) q^n + 2 \sum_{\left( \frac{-r n}{\ell} \right)= -1} p_r \left(n - r \left(\frac{\ell^2-1}{24}\right) \right) q^n \right) \pmod{\ell} \\
 &\equiv \eta^{\ell^2 r} ((f_{r,\ell,0})^{\ell} + 2 f_{r,\ell,-1})  \pmod{\ell}.
\end{align*}
Now we have
\begin{equation}\label{eq:congforlemma}
2 f_{r,\ell,-1} \equiv  - (f_{r,\ell,0})^{\ell}  + \eta^{-r\ell^2} (  \Delta^{\frac{r (\ell^2-1)}{24}}  + \left(\frac{-r}{\ell} \right)  \Theta^{\frac{\ell -1}{2}} \Delta^{r \left( \frac{ \ell^2 -1}{24} \right)}) \pmod{\ell}.
\end{equation}

From Lemma \ref{thm:genfcn0}, $f_{r,\ell,0}^{\ell}$ is congruent to a form in $f_{\frac{r\ell(\ell^2-1)}{2} - \frac{r\ell^2}{2}}(\nu_{\eta}^{-r})$.
By the second part of Lemma \ref{thm:filtration}, we have
$$
w\left( \Theta^{\frac{\ell -1}{2}} \Delta^{r \left( \frac{ \ell^2 -1}{24} \right)} \right) =  \frac{ r(\ell^2 -1)}{2} + \frac{(\ell -1)}{2} (\ell+1) = \frac{(r+1)(\ell^2-1)}{2}.
$$

Since $E_{\ell-1} \equiv 1 \pmod{\ell}$, the weight of any term in (\ref{eq:congforlemma}) can be increased by a multiple of $\ell -1$.  It follows that the right hand side of \ref{eq:congforlemma} is congruent to a weakly holomorphic form of weight $\frac{r \ell (\ell^2-1)}{2} - \frac{r\ell^2}{2} = k_{r,\ell,-1}$.
\end{proof}

The following result will also be useful. 
\begin{lem}\label{thm:genfcn-1}
We assume $\ell \ge 5$ is prime, and that $r$ is an odd integer co-prime to $\ell$.  Additionally, we assume that for $\delta = 0$ and $\delta=1$, either $f_{r,\ell,\delta} \equiv 0 \pmod{\ell}$ or $\ord_{\ell}(f_{r,\ell,\delta}) > 0$. Then $(f_{r,\ell,0})^{\ell} +2 f_{r,\ell,-1}$ is congruent to a modular form in $S_{\frac{\ell^2 - r -1}{2}}(\nu_{\eta}^{-r })$. 
\end{lem}
\begin{proof}
The previous proof shows
$$
(f_{r,\ell,0})^{\ell} +2 f_{r,\ell,-1} \equiv \eta^{-r\ell^2} ( E_{\ell-1}^{\frac{\ell+1}{2}} \Delta^{\frac{r (\ell^2-1)}{24}}  -\left(\frac{-r}{\ell} \right)  \Theta^{\frac{\ell -1}{2}} \Delta^{r \left( \frac{ \ell^2 -1}{24} \right)}) \pmod{\ell} . 
$$
Using the assumptions on the order of vanishing of $f_{r,\ell,0}$ and $f_{r,\ell,-1}$ modulo $\ell$, the Fourier expansion of the right hand side is supported only on positive indices, so the order of vanishing of $E_{\ell-1}^{\frac{\ell+1}{2}} \Delta^{\frac{r (\ell^2-1)}{24}}  - \left(\frac{-r}{\ell} \right)  \Theta^{\frac{\ell -1}{2}} \Delta^{r \left( \frac{ \ell^2 -1}{24} \right)}$ modulo $\ell$ is greater than $r \ell^2/24$.

Using the Miller basis, we obtain that $ E_{\ell-1}^{\frac{\ell+1}{2}} \Delta^{\frac{r (\ell^2-1)}{24}} -  \left(\frac{-r}{\ell} \right)  \Theta^{\frac{\ell -1}{2}} \Delta^{r \left( \frac{ \ell^2 -1}{24} \right)}$ is congruent to a form in $S_{\frac{(r+1)(\ell^2-1)}{2}}$ with order of vanishing greater than $\frac{r \ell^2}{24}$. It follows that $(f_{r,\ell,0})^{\ell} +2 f_{r,\ell,-1} $ is congruent modulo $\ell$ to a form in $S_{\frac{(r+1)(\ell^2-1)}{2} - \frac{r \ell^2}{2}} (\nu_{\eta}^{-r \ell^2 })= S_{\frac{\ell^2-r-1}{2}} (\nu_{\eta}^{-r})$. 
\end{proof}

\section{Theorem \ref{thm:0class}}
Theorem \ref{thm:scarcity} relies on the restriction given by Theorem \ref{thm:0class}, so we prove it first.
\begin{proof}

We follow the method of Proposition 3.2 of \cite{ABR}. We set $m' = Q_1 \cdots Q_s$ for some primes $Q_j$. Let $\delta = \lp  \frac{r (r-24t)}{\ell} \rp$, and set $\epsilon_i := \lp \frac{r(r -24t)}{Q_i} \rp$ for $1 \le i \le s $. 

For $1 \le i \le s$, let $Y_i$ be the operator given by
$$
Y_i := (1 + (\mbox{ } \otimes \chi_{Q_i}) \epsilon_i - U_{Q_i}V_{Q_i}). 
$$
We let 
$$
g := f_{r, \ell, \delta} | Y_1 Y_2 \cdots Y_s.
$$
Then (\ref{eq:umdef}) and (\ref{eq:twistdef}) imply
$$
g \equiv \sum_{\substack{  \lp  \frac{- rN}{\ell} \rp = \delta \\  \lp  \frac{- rN}{Q_j} \rp = \epsilon_j, 1 \le j \le s }} p_r \left(\frac{N + r}{24} \right) q^{\frac{N}{24}} \pmod{\ell}. 
$$
It follows from \ref{eq:uqaction} and \ref{eq:quadtwist} that $g \in M_{k_{r,\ell,\delta}}^! ( m'^2, \nu_{\eta}^{-r})$, with $k_{r,\ell,\delta}$ as in Propositions \ref{thm:genfcn0} and \ref{thm:genfcn1}.  By Proposition \ref{thm:squareclasses}, the congruence (\ref{eq:cong8}) is equivalent to the congruence
$$
g \big| U_{m''} \equiv 0 \pmod{\ell}. 
$$

Since $r$ is odd, it follows from Lemma 3.3 of \cite{ABR} that $g \big| U_{m''} \equiv 0 \pmod{\ell}$ holds if and only if $g \equiv 0 \pmod{\ell}$.  Note that Lemma 3.3 is stated with the assumptions that $(r, 3) =1$ and that $f$ is a holomorphic modular form, but the proof never uses either. 
\end{proof}

\section{Theorem \ref{thm:scarcity}}
This is a finite calculation, but due to the rapid growth of the functions $p_r(n)$, the time required for the larger values $\ell$ and $m$ would obstruct a brute-force approach. We modify an algorithm used in \cite{ABR} and provide a description here for completeness. The Mathematica \cite{Mathematica} code used is on the first author's webpage, see the files ``computepartitions.nb" and ``rulingoutcongruencese.nb".

For each odd $r$ between $1$ and $24$, $10^5$ values of $p_r(n)$ were computed using the following recursive formula (see \cite{MR151446}) :
    $$p_r(n) = \frac{r}{n}  \sum_{j=0}^{n-1}p_r(j) \sigma (n-j),$$
where $\sigma (n-j)$ is the sum of positive divisors of $n-j$. 

Fix a positive integer $1 < r <24$ and $\delta \in \{0,-1\}$. We verify the claim of Theorem \ref{thm:scarcity} for $r$ and $\delta$ as follows:

For each prime $5 \le \ell \le 6133$, first we check whether there is a Ramanujan-type congruence for $\ell$ and $r$. Then to search for theta-type congruences, we compute a list $tvalues$ of positive integers $t$ such that $\left( \frac{r (r-24t)}{\ell} \right)= \delta$ and $p_r( t) \not\equiv 0 \pmod{\ell}$. Fsor each prime $m$ such that $5 \le m \le 6133$, check that $tvalues$ contains at least one value $t_{\epsilon}$ with $\left( \frac{r (r-24t_{\epsilon})}{m} \right)= \epsilon$ for both $\epsilon = 1$ and $\epsilon = -1$. If it does, then by Theorem \ref{thm:0class} and Proposition \ref{thm:squareclasses}, we can conclude that (\ref{eq:squarefree}) does not hold for this value of $m$ with any $t$ such that $\left(\frac{r(r-24t)}{\ell} \right) = \delta$. 


These steps were carried out for each $r $ and $\delta \in \{ 0,1 \}$. For each $r$, the only instances of (\ref{eq:squarefree}) are those given in the table. 

\section{Theorems \ref{thm:etafamily}, \ref{thm:main}, and \ref{thm:abnormal}}

\subsection{Spaces of cusp forms}

\begin{lem}\label{thm:1dspaces}
Let $k \le 24$ and $n < 24$ be non-negative integers. Then
 \begin{equation*}
    S_{\frac{k}{2}}(\nu_\eta^n) \equiv \begin{cases} \{0\} & n\not \equiv k \pmod{4} \mbox{ or }  n >  k \mbox{ or } k=n+4 \\
     \mathbb{C} \eta^k(z) & k =n \\
     \mathbb{C}  E_{\frac{k-n}{2}}\eta^n(z): &k > n+4  \mbox{ and }  k\equiv n \pmod{4}  \end{cases}
\end{equation*}
\end{lem}
\begin{proof} Let $r=12+\frac{k-n}{2}$, then $24\geq r \geq 0$. For any $f \in S_{\frac{k}{2}}(\nu_\eta^n)$, we have  $f\eta^{24-n}(z)\in S_{r}$. If
$n \not\equiv k \pmod{4}$, then $r$ must be odd, in which case $S_r  = \{0 \}$. If $n > k$ then $r < 12$, which implies that $S_r  = \{0 \}$, and similarly when $k=n+2$ we have $S_r = S_2 = \{0\}$.  When $k = n$, $S_r  =  S_{12} = \mathbb{C} \eta^{24}$.  For the last case, suppose $k \equiv n \pmod{4}$ and $k >n$. Then $S_r$ is one dimensional by (\ref{eq:dimensionformula}), and we have  $S_r = \mathbb{C} E_{r -12}\Delta$. 

\end{proof}

\subsection{Proof of Theorem \ref{thm:etafamily}}
First we prove case (1). From Lemma \ref{thm:genfcn0}, we have
$$
f_{r,\ell,0} \equiv  \frac{\Delta^{\frac{r(\ell^2 - 1)}{24}} | U_{\ell}}{\eta^{r\ell}} \pmod{\ell}.
$$

First we explain why our assumptions imply that $f_{r,\ell,0}$ is congruent to a cusp form. The order of vanishing of the numerator is at least $\lceil \frac{ r(\ell^2-1)}{24 \ell} \rceil$, and the order of vanishing of the denominator is $\frac{r \ell }{24}$. We have
\begin{align*}
\Big\lceil \frac{ r(\ell^2-1)}{24 \ell} \Big\rceil &= \Big\lceil \frac{r \ell}{24} - \frac{r}{24\ell} \Big\rceil \\
&=  \Big\lceil \frac{r \ell}{24} - \frac{a}{24} +  \frac{a +1}{24\ell} \Big\rceil \\
&= \Big\lceil \frac{r\ell -a}{24} +  \frac{a +1}{24\ell} \Big\rceil \\
&\ge \Big\lceil \frac{r\ell}{24} \Big\rceil \\
&> \frac{r\ell}{24} .
\end{align*}

In the last two steps, we used the assumptions that $r \ell \equiv 23 \pmod{24}$ and $0 < a < 23$.

By Lemma \ref{thm:filtration}, the form $\Delta^{\frac{r(\ell^2 - 1)}{24}} | U_{\ell}$ is congruent to a form in $S_w$, where \begin{equation}\label{eq:wformula}
w = (\ell -1) \Big\lfloor \frac{1}{\ell -1} \cdot \left( \ell + \frac{r(\ell^2-1) - 2}{2\ell} \right) \Big\rfloor .
\end{equation} 

We compute $w$:

\begin{align*}
 \Big\lfloor \frac{1}{\ell -1} \cdot \left( \ell + \frac{r(\ell^2-1) - 2}{2\ell} \right) \Big\rfloor  &=  \Big\lfloor \frac{\ell}{\ell -1} + \frac{r (\ell+1)}{2 \ell} - \frac{1}{\ell(\ell-1)} \Big\rfloor \\
&= \Big\lfloor \frac{\ell^2-1}{\ell(\ell -1)} + \frac{r}{2} + \frac{r}{2\ell} \Big\rfloor \\
&= \Big\lfloor \frac{\ell + 1}{\ell } + \frac{a(\ell-1)}{2} - \frac{1}{2} + \frac{a}{2} - \frac{(a+1)}{2 \ell} \Big\rfloor \\
&= 1 +  \frac{a(\ell-1)}{2} + \frac{a-1 }{2} + \Big\lfloor \frac{1}{\ell }   - \frac{(a+1)}{2 \ell} \Big\rfloor \\
&= \frac{a\ell + 1}{2} + \Big\lfloor \frac{1-a}{2\ell} \Big\rfloor \\
&= \frac{a \ell - 1}{2}.
\end{align*}
In the last step, we used the assumption that $\ell \ge \frac{a-1}{2}$.

Now we have $4 | w$, in particular, $w \not\equiv 2 \pmod{12}$, so $\operatorname{dim} (S_w) = \Big\lfloor \frac{w}{12} \Big\rfloor$. We compute
$$
\Big\lfloor \frac{w}{12} \Big\rfloor = \Big\lfloor \frac{a \ell^2 - a \ell - \ell + 1}{24} \Big\rfloor 
$$
and 
$$
\frac{r \ell}{24} = \frac{a\ell^2 - a \ell - \ell}{24}.
$$
From $a \ell^2 - a \ell - \ell + 1 \equiv 0 \pmod{24}$ and $r \ell /24 \not\in \mathbb{Z}$, we deduce
$$
\Big\lfloor \frac{w}{12} \Big\rfloor > \frac{r \ell}{24}.
$$
Using the Miller basis, there exists a form in $S_w$ congruent to $\Delta^{\frac{r(\ell^2 - 1)}{24}} | U_{\ell}$ with order of vanishing greater than $\frac{r \ell}{24}$. 

We now have $f_{r,\ell,0} \in S_{w - \frac{r\ell}{2}} (\nu_{\eta}^{-r\ell}) = S_{w - \frac{r\ell}{2}} (\nu_{\eta})$. It follows from the calculation of $w$ above that $w - \frac{r \ell}{2} = \frac{1}{2}$.

Thus we've shown $f_{r,\ell,0} \in S_{\frac{1}{2}} (\nu_{\eta})$. By Lemma \ref{thm:1dspaces}, we have $f_{r,\ell,0} \equiv \alpha \eta \pmod{\ell}$ for some integer $\alpha$. Comparing the coefficients of $q^{1/24}$ in the Fourier expansions of $\eta$ and $f_{r,\ell,0}$, we deduce that we can take $\alpha = p_r( \frac{\ell + r}{24} )$.

The second case is basically identical to the first case. Once again, we have
$$
f_{r,\ell,0} \equiv  \frac{\Delta^{\frac{r(\ell^2 - 1)}{24}} | U_{\ell}}{\eta^{r\ell}} \pmod{\ell}.
$$
We can once again compute that under these assumptions, $f_{r,\ell,\delta}$ is congruent to a cusp form in $S_{w - (r\ell)/2}( \nu_{\eta}^{-r\ell})$, where $w$ is once again given by (\ref{eq:wformula}). Our assumptions imply that $-r \ell \equiv 3 \pmod{24}$, so that $f_{r,\ell,0} \in S_{\frac{3}{2}}( \nu_{\eta}^3)$.

\subsection{Proof of Theorem \ref{thm:main}}
Assume $r = r_1, \ell$ are as in the first and second columns of Table \ref{Tab:etacongs}. Then it is easy to verify that $r$ and $\ell$ satisfy the requirements of Theorem \ref{thm:etafamily}, so by Theorem \ref{thm:etafamily} and (\ref{eq:etaidentity}), we have
$$
\sum_N p_r \lp \frac{\ell N +r}{24} \rp q^{\frac{N}{24}} \equiv \alpha \sum_{n=1}^{\infty} \left( \frac{12}{n} \right) q^{n^2/24}.
 $$
For any $t$ for which $\ell \nmid p_r (t)$, we must have
 $$
 t = \frac{\ell n^2 +r }{24}
 $$
 for some integer $n$. For every prime $m \neq \ell$, we have
 $$
 -1 \neq \left( \frac{n^2}{m} \right) = \left( \frac{\ell}{m} \right) \left(\frac{24t -r}{m} \right).
 $$
On the other hand, if $-1 = \left( \frac{\ell}{m} \right) \left(\frac{24t -r}{m} \right)$, then we have $p_r(n) \equiv 0 \pmod{\ell}$ for all $n \equiv t \pmod{\ell m}$.

The other case is essentially identical. Assume $r = r_2, \ell$ are as in the first and third columns of Table \ref{Tab:etacongs}.Then by Theorem \ref{thm:main} and (\ref{eq:eta3identity}), we have
$$
\sum_N p_r \lp \frac{\ell N +r}{24} \rp q^{\frac{N}{24}} \equiv \alpha \sum_{n \ge1} \left( \frac{-4}{n} \right)n q^{3n^2/24} $$
for some integer $\alpha$. For any $t$ of the form $\frac{\ell n + r}{24}$ for which $\ell \nmid p_r (t)$, we must have
 $$
 t = \frac{3 \ell n^2 +r }{24}
 $$
 for some integer $n$. For any prime $m \nmid 3 \ell$, we have
 $$
 -1 \neq \left( \frac{n^2}{m} \right) = \left( \frac{3 \ell}{m} \right) \left(\frac{24t -r}{m} \right).
 $$
Thus, if $-1 = \left( \frac{3 \ell}{m} \right) \left(\frac{24t -r}{m} \right)$, then we have $p_r(n) \equiv 0 \pmod{\ell}$ for all $n \equiv t \pmod{\ell m}$.

 \subsection{Proof of Theorem \ref{thm:abnormal}}
 \ \\

\textbf{Case (1):} By Lemma \ref{thm:genfcn0} and Lemma \ref{thm:filtration}, we have $f_{r,\ell,0}$ congruent modulo $\ell$ to a form in $S_{\ell/2} (\nu_{\eta}^{-r\ell})$.  Let $\ell_0$ be the smallest non-negative integer with $\ell_0 \equiv \ell \pmod{24}$. Then we have $f_{r,\ell,0} \eta^{24 - \ell_0} \in S_{12 + \frac{\ell - \ell_0}{2}} $. The dimension of $S_{12 + \frac{\ell - \ell_0}{2}} $ is $d=1 + \frac{\ell-\ell_0}{24}$ by (\ref{eq:dimensionformula}). We write $f_{r,\ell,0} \eta^{24 - \ell_0}$ in terms of the Miller basis $f_1, f_2, \cdots, f_d$ for $S_{12 +  \frac{\ell - \ell_0}{2}}$:
$$
f_{r,\ell,0} \eta^{24 - \ell_0} \equiv \sum_{i=1}^d \alpha_i f_i \pmod{\ell}.
$$
By the condition that $ \operatorname{ord}_{\ell}(f_{r,\ell,0}) > \ell$, we have $\alpha_i \equiv 0 \pmod{\ell}$ for $1 \le i \le d-1$, so 
$$
f_{r,\ell,0} \eta^{24 - \ell_0} \equiv \alpha_d f_d \pmod{\ell}.
$$
We note that $f_d = \Delta^{1 +  \frac{\ell - \ell_0}{24}}$, and therefore $f_{r,\ell,0} \equiv \alpha_d \eta^{\ell} \pmod{\ell}$.

\textbf{Case (2):}  The conditions in Column A and Lemma \ref{thm:genfcn-1} guarantee  $ f_{r,\ell,0}^{\ell} + 2 f_{r,\ell,-1} \in S_{3 /2} (\nu_{\eta}^{-r})$. Since $f_{r, \ell, 0} \equiv 0 \pmod{\ell}$, we have $f_{r,\ell,-1} \in S_{3  /2} (\nu_{\eta}^{-r})$. The result follows from Lemma \ref{thm:1dspaces} and the assumption that $r \equiv -3 \pmod{24}$.

\textbf{Case (3):}
By Lemma \ref{thm:genfcn-1}, $2f_{r,\ell,-1} + f_{r,\ell,0}^{\ell} \in  S_{\frac{1}{2}} ( \nu_{\eta})$. By Lemma \ref{thm:1dspaces} we have
$$
2 f_{r,\ell,-1} + f_{r,\ell,0}^{\ell}  \equiv \alpha \eta \pmod{\ell}
$$
for some $\alpha \in \mathbb{Z}$, and by the assumption on $f_{r,\ell,0}$, we have
$$
2 f_{r,\ell,-1}  \equiv \beta \eta^{\ell^2}- \alpha \eta \pmod{\ell}.
$$
Since $f_{r,\ell,-1}$ is not supported on any multiples of $\ell$, we must have $\alpha \equiv \left( \frac{12}{\ell} \right) \beta \pmod{\ell}$, and therefore $f_{r,\ell,-1} \equiv \alpha \left( \left( \frac{12}{\ell} \right) \eta^{\ell^2} - \eta\right) \pmod{\ell}$.

\section{Examples and Future Directions}
Examples of congruences arising from Theorem \ref{thm:main} are straightforward to compute. In both cases of Theorem \ref{thm:main}, each of the pairs $(r,\ell)$ lie on one of eleven lines. This is shown for the first case of Theorem \ref{thm:main} in Figure \ref{fig:etaexamples}. 

\begin{figure}[h]
\caption{The 66 pairs $(r,\ell) \in [1, 501] \times [5,1583]$ in the first case of Theorem \ref{thm:etafamily}}\label{fig:etaexamples}
\includegraphics[width=8cm]{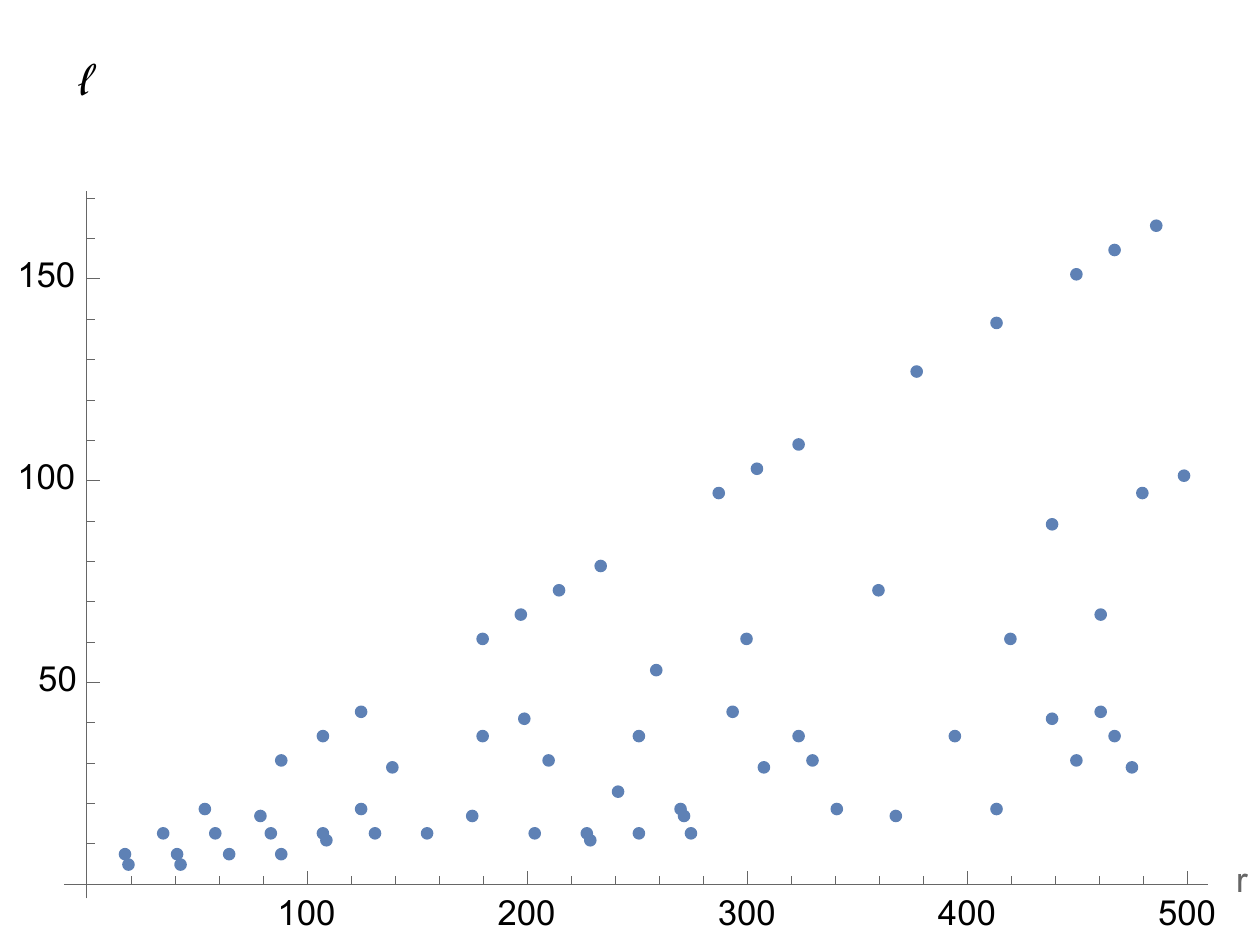}

\end{figure}

The congruences in Theorem \ref{thm:abnormal} are more elusive. Examples of case (1) of Theorem \ref{thm:abnormal} exhibit a degree of regularity - there appears to be exactly one pair of congruences for each $r \equiv 23 \pmod{24}$. Only a few examples of cases (2) and (3) - those given in Table \ref{Tab:abnormalexamples} - have been found after searching over $(r,\ell) \in [1, 701] \times [5,1583]$, and it is unclear whether these are the only examples of these types. We note also that there are examples of theta-type congruences that are not covered by any of our theorems, for example, one can use Lemma \ref{thm:genfcn-1} and properties of the space $S_{25/2} (\nu_{\eta})$ to show that $f_{23,7,-1}$ is a multiple of $\eta^{7^2} - \left( \frac{12}{7}\right) \eta$ modulo $7$, similar to Theorem \ref{thm:abnormal} (3). We leave the problem of better understanding these examples open for future investigation.




\bibliographystyle{amsalpha}
\bibliography{part_cong}
\end{document}